\documentclass[a4paper,10pt,reqno]{amsart}
\usepackage{amsfonts,amsmath,amssymb,amsthm,amsaddr,datetime}
\usepackage{graphicx}
\usepackage{fixmath}
\usepackage{hyperref}
\usepackage{tikz}
\usepackage{mathrsfs}
\usepackage{mathtools}
\usepackage{stmaryrd}
\usepackage{psfrag}
\usepackage{epsf}
\usepackage{url,rotating}
\usepackage[cp1250]{inputenc}
\usepackage{amsmath,amsfonts,amssymb,latexsym}
\usepackage{setspace}
\usepackage{enumitem}
\usepackage{multirow}
\usepackage[margin=1.2in]{geometry}
\usepackage{makecell} 

\newtheorem{thm}{Theorem}
\newtheorem{conj}{Conjecture}

\newtheorem{lemma}{Lemma}

\newtheorem{corollary}[thm]{Corollary}

\newenvironment{kst}
{\setlength{\leftmargini}{2\parindent}
 \begin{itemize}
 \setlength{\itemsep}{-1.1mm}}
{\end{itemize}}

\begin{document}


\pagestyle{plain}
\setlength{\baselineskip}{15pt}
\title{Every graph is uniform-span $(2,2)$-choosable:\\
 Beyond the 1-2 conjecture}

\author{Kecai Deng, Hongyuan Qiu}

\begin{center}

\end{center}
\address{
School of Mathematical Sciences, Huaqiao University, \\
              Quanzhou 362000, Fujian, P.R. China
}

\email{\qquad  kecaideng@126.com {\bf (K. Deng)}, hongyuan0.qiu@gmail.com {\bf (H. Qiu)}}
\thanks{Kecai Deng is the corresponding author. \\
\hspace*{4.3mm}Kecai Deng is partially supported by NSFC (No. 11701195),   Fundamental Research Funds for the Central
Universities (No. ZQN-904) and Fujian Province University Key Laboratory of Computational Science,
School of Mathematical Sciences, Huaqiao University, Quanzhou, China (No. 2019H0015).}

\makeatletter
\@namedef{subjclassname@2020}{\textup{2020} Mathematics Subject Classification}
\makeatother

\begin{abstract} For a simple graph $G=(V,E)$, a \emph{proper total weighting} is a mapping  $w: V\cup E\rightarrow \mathbb R$ such that for every edge $uv\in E$, $w(u)+\sum_{e\ni u}w(e)\neq w(v)+\sum_{e\ni v}w(e)$. The graph  $G$ is said   $(2,2)$-\emph{choosable} if,   for any list assignment
$L$ that assigns to each    $z$ in $V\cup E$ a set $L(z)$ of two real numbers,
there exists a {proper total weighting} $w$ with $w(z)\in L(z)$ for every $z\in V\cup E$. Wong and Zhu,  and independently    Przyby{\l}o
and Wo\'{z}niak conjectured  that   every simple graph is $(2,2)$-choosable. This conjecture remains open.

 For a  set $\{a,b\}\subset \mathbb R$, its span is defined as  $|b-a|$.  We call a graph $G=(V,E)$   \emph{uniform-span} $(2,2)$-\emph{choosable} if,   for any list assignment
$L$ that assigns to every  $z\in V\cup E$ a two-element list of a common   span,
there exists a {proper total weighting}    respect to    the assignment. In this paper, we present a novel lemma  and perform comprehensive enhancements to our previous algorithm. These contributions enable us to prove   that every graph is  uniform-span $(2,2)$-choosable. This  confirms the 1-2 conjecture in   full generality, and provides   supporting evidence for the $(2,2)$-choosable conjecture.
\end{abstract}

\maketitle
\section{Introduction}
In this paper, we consider simple, undirected, and finite graphs. For   notation or terminology not explicitly defined here,   refer to \cite{Bondy}.

Let  $G=(V,E)$ be a graph. An \emph{edge weighting} is a mapping  $w: E \rightarrow \mathbb{R}$;   the \emph{weighted degree} of a vertex $v$   is defined as  $\sigma_w(v) =
 \sum_{e \ni v} w(e)$. The weighting  $w$   is  \emph{irregular}   if all vertices   receive   distinct weighted degrees, and  \emph{proper} if   adjacent vertices always differ in theirs. When  the image  of $w$ is restricted to a subset $Q\subseteq \mathbb{R}$ we speak of a \emph{$Q$-edge weighting}, and    a  $\{1,2,\ldots,k\}$-edge weighting  is also called  a \emph{$k$-edge weighting}. The \emph{irregular strength} $s(G)$ is the smallest  $k$ for which such an irregular  $k$-edge weighting exists,   whereas  $\tau_0(G)$ denotes the least    $k$ admitting a  proper $k$-edge weighting.

 Proceeding  analogously, a  \emph{total weighting}    is a mapping $w: V \cup E \rightarrow \mathbb{R}$  and the \emph{weighted degree} of  $v$  becomes   $\sigma_w(v) = w(v) + \sum_{e \ni v} w(e)$.  The same notion of  properness applies.
A \emph{$Q$-total weighting}  restricts weights to $Q\subseteq \mathbb{R}$ and a \emph{$k$-total weighting}  restricts weights to  $\{1,2,\ldots,k\}$.   Denote by $\tau(G)$ is the least  $k$ for which $G$ admits a proper  $k$-total  weighting.

The irregularity strength was first introduced by Chartrand, Jacobson, Lehel, Oellermann, Ruiz, and Saba \cite{Chartrand}. Faudree and Lehel's influential conjecture  \cite{Faudree} served as a catalyst for significant developments in the investigation of irregularity strength.

\begin{conj} \label{conj01}  \cite{Faudree}   There is a constant $C>0$  such that for all $d$-regular graphs $G$ on $n$ 	vertices and with $d>1$, $s(G)\leq \frac{n}{d}+C$.
\end{conj}
\noindent Although this
conjecture  itself  remains open,   Przyby{\l}o and Wei   \cite{Przybylo6,Przybylo7}    have recently confirmed its asymptotic validity. In 2011,
Kalkowski, Karo\'{n}ski and Pfender \cite{Kalkowski2011} provided a concise proof  that  every
nice graph (i.e., a graph without isolated edges) with minimum degree  $\delta\geq1$ satisfies     $s(G)\leq 6\lceil\frac{n}{\delta}\rceil$, and this bound remains     the best known result  for general nice graphs. Conjecture \ref{conj01} has not only laid the theoretical foundation for an entire research discipline but has also inspired numerous related studies, novel concepts, and intriguing questions. Recent examples include investigations into highly irregular spanning subgraphs  by    Alon and Wei \cite{AlonWei}, Fox, Luo and Pham \cite{FoxLuoPham}, Lu\v{z}ar,   Przyby{\l}o and Sot\'{a}k \cite{Luzar2025}, as well as Ma and Xie \cite{MaXie}.    A comprehensive bibliography on irregularity strength and Conjecture \ref{conj01}   can be found in    \cite{Przybylo7}.

 As the local variant of irregularity strength problem,       Karo\'{n}ski, {\L}uczak and
Thomason   \cite{Thomason}   proposed the following problem concerning   proper $k$-edge  weighting, known as the 1-2-3 conjecture.

 \begin{conj} \label{conj1} \cite{Thomason} $($1-2-3 Conjecture$)$ Every   nice graph admits a proper $3$-edge weighting. \end{conj}

 \noindent Before its resolution, this conjecture garnered significant attention due to its elegance. For every nice      graph $G$,  Addario-Berry, Dalal, McDiarmid, Reed and
Thomason \cite{AddarioBerry}  initially  showed that $\tau_0(G)\leq30$. This upper bound was subsequently  reduced to   16 by Addario-Berry, Dalal and Reed \cite{AddarioBerry1},
to 13 by Wang and Yu \cite{Wang} and   to 5    by        Kalkowski, Karo\'{n}ski and Pfender \cite{Kalkowski}. For  specific  graph classes,  Zhong \cite{Zhong} demonstrated  that  there exists a threshold   $n'$ such that every graph $G$ with $n\geq n'$ vertices  and     $\delta> 0.99985n$ satisfies  $\tau_0(G)\leq3$. Przyby{\l}o proved the conjecture   for $d$-regular graphs with $d\geq10^8$ \cite{Przybylo}, and  for graphs  with $\delta=\Omega(\log \Delta)$ \cite{Przybylo1}. Additionally, it was shown that $\{1,2\}$  are asymptotically almost surely sufficient for a random graph  to admit  a proper edge weighting \cite{AddarioBerry1}. Regarding  algorithmic complexity,
    Dudek and Wajc  \cite{Dudek} proved  that determining whether a particular graph admits a proper 2-edge weighting is NP-complete, while Thomassen, Wu and Zhang  \cite{Thomassen1}   showed that this problem is polynomial-time solvable for bipartite graphs.   Recently, Keusch achieved two major breakthroughs, first lowering the upper bound to 4 \cite{Keusch},  then to 3  \cite{Keusch1}, thereby resolving the conjecture affirmatively.

 Building on the famous 1-2-3 conjecture, the 1-2 conjecture concerning proper $k$-total weighting was proposed by   Przyby{\l}o and Wo\'{z}niak  \cite{Przybylo2} in 2010, which  is not only a variant of Conjecture \ref{conj1}, but has also inspired methodological advances that subsequently led to improved bounds for several conjectures, including   Conjecture \ref{conj01}  and Conjecture \ref{conj1} itself.
 \begin{conj} \label{conj2} \cite{Przybylo2} $($1-2 Conjecture$)$  Every   graph admits a proper $2$-total weighting. \end{conj}
 \noindent Initial progress on this conjecture showed that $\tau(G)\leq11$ for each graph $G$ \cite{Przybylo2}.  Since a graph can admit a 1-total weighting if and only if every pair of adjacent vertices has distinct degrees,   Przyby{\l}o noted in \cite{Przybylo3} that regular graphs present the most challenging case. Indeed, Przyby{\l}o  himself \cite{Przybylo3} improved the bound to
$\tau(G)\leq7$ for each regular graph $G$.
  Later,   Kalkowski \cite{Kalkowski1} demonstrated
that every graph $G=(V,E)$ admits a proper total weighting $w$ with $w(v)\in\{1,2\}$ for
each $v\in V$ and $w(e)\in\{1,2,3\}$ for each $e\in E$. This result was subsequently implied by Keusch's resolution \cite{Keusch1} of the 1-2-3 conjecture, and Wong and Zhu's work \cite{Wong1} on the list version of the 1-2 conjecture. In our recent work  \cite{dq}, we established   that every graph  admits a proper $\{0,1\}$-total weighting, through a simple and well-designed algorithm. This   provided the first example of a two-element set that guarantees proper total weighting for all graphs.  As a direct consequence,  the 1-2 conjecture was confirmed  for all regular graphs.

Notably,  Kalkowski's algorithmic approach \cite{Kalkowski1} developed for the 1-2 conjecture enabled  Kalkowski, Karo\'{n}ski and Pfender \cite{Kalkowski} to   dramatically improve the general bound for the 1-2-3 conjecture from 16 to 5. This approach has also been adapted to establish important  results for the hypergraph   generalization  \cite{Kalkowski2017}  of the 1-2-3 conjecture, and to derive  the currently best known  bounds on the irregularity strength of general nice graphs \cite{Kalkowski2011}.

Furthermore, the list variants of both the 1-2-3 conjecture  and the 1-2 conjecture  exhibit deep theoretical connections. We say a graph $G=(V,E)$ is \emph{edge-weight $k$-choosable} if for any list assignment
$L$ that assigns to each edge $e$ a set $L(e)$ of $k$ real numbers,
there exists a  proper edge weighting $w$ such that $w(e)\in L(e)$ for each $e\in E$. A list variant  of the 1-2-3 conjecture  was proposed by  Bartnicki,
Grytczuk and Niwczyk \cite{Bartnicki09}, which remains open.

\begin{conj} \label{conj3} \cite{Bartnicki09}   Every nice graph is edge-weight $3$-choosable.  \end{conj}

\noindent We say a graph $G=(V,E)$ is \emph{$(k, k')$-choosable} if for any list assignment
$L$ that assigns to each vertex $v$ a set $L(v)$ of $k$ real numbers and to each edge $e$ a set $L(e)$ of $k'$ real numbers,
there exists a  proper total weighting $w$ such that $w(z)\in L(z)$ for each $z\in V\cup E$. Wong and Zhu   \cite{Wong}, and independently   Przyby{\l}o
and Wo\'{z}niak   \cite{Przybylo4} proposed the following   list version of the 1-2 conjecture.
\begin{conj} \label{conj4} \cite{Przybylo4,Wong}  Every graph is $(2, 2)$-choosable.  \end{conj}

\noindent As a strengthening  of both the 1-2-3 conjecture and  Conjecture \ref{conj3}, Wong and Zhu  \cite{Wong} proposed the following conjecture.

\begin{conj} \label{conj5} \cite{Wong} Every nice graph is $(1, 3)$-choosable.  \end{conj}

 \noindent  Significant  progress  has been made on Conjectures \ref{conj4} and \ref{conj5}.    In 2016,   Wong and Zhu \cite{Wong1} showed that    every graph is $(2, 3)$-choosable. In 2021, Cao \cite{Cao} showed that every nice graph is $(1, 17)$-choosable. Recently, by developing the method from   \cite{Cao},    Zhu \cite{Zhu}   showed  that every nice  graph is $(1, 5)$-choosable.

 For more problems and conjectures related to the 1-2-3 conjecture and the 1-2  conjecture, we refer the reader to \cite{Grytczuk,Seamone,Bensmail20231,Baudona} and the references cited there.

For a  set $\{a,b\}\subset \mathbb R$, its span is defined as  $|b-a|$.  We call a graph $G=(V,E)$   \emph{uniform-span} $(2,2)$-\emph{choosable} if,   for any list assignment
$L$ that assigns to every  $z\in V\cup E$ a two-element list of a common   span,
there exists a {proper total weighting}    respecting   the assignment.
In this paper, we present a novel lemma (see Lemma \ref{lemma1} in Section 2) and perform comprehensive enhancements to our previous algorithm from \cite{dq}. These contributions enable us to  prove  the following  statement, which  confirms the 1-2 conjecture in full generality and   provides   supporting evidence for the $(2,2)$-choosable conjecture.

   \begin{thm} \label{thm}
 Every graph is  uniform-span $(2,2)$-choosable.
\end{thm}

  \begin{corollary} \label{corollary}
 Every graph  admits  a proper $\{a,b\}$-total weighting, given arbitrary two distinct real numbers $a,b$; specially, every graph  admits  a proper $2$-total weighting.
\end{corollary}

\noindent The proof of     Theorem \ref{thm} is presented in Section 2.

Our methods and techniques might provide potential help for the related research topics we mentioned before.   Finally, we present the    following  question that lies   between the 1-2-3 conjecture and Conjectures \ref{conj3} and \ref{conj5}.

\begin{conj} \label{conj6} Every nice graph admits  an $\{a,b,c\}$-edge weighting, for arbitrary given pairwise distinct real numbers $a,b,c$.  \end{conj}

\section{Proof of Theorem \ref{thm}}

Let $G=(V,E)$ be a  graph.  For  $X,Y\subseteq V$ with $X\cap Y=\emptyset$, we denote   by $G[X,Y]$ the bipartite subgraph of $G$ with vertex set   $X\cup Y$ and edge set consisting of all    edges in $G$ that have one endpoint in   $X$ and the other in $Y$. For integers $k$ and $l$ with $k\leq l$, we   denote by $[k,l]$ the set of integers   $\{k,k+1,\ldots,l\}$.

 Let $\phi$ be a vertex coloring (not necessarily  proper) of   $G$ using non-negative integers. An independent set $I$ of $G$ is called
 \emph{{$\phi$-{maximum}}} if it maximizes $\sum_{v\in I}\phi(v)$  and each vertex in $V\setminus I$ has a neighbor in $I$.  

 Let $G=(W\cup U,E)$ be a bipartite graph where $W$ and $U$ are the two disjoint independent sets.  A subgraph $F$  of $G$ is called   a    \emph{$U$-star-covering}, if it covers each vertex of $U$ exactly once. That is, each component of $F$ is a star, where each vertex in $U$ is a leaf. Let   $\phi$ be a vertex coloring (not necessarily  proper) of   $G$ using positive integers. A vertex $x$ in $V_F\setminus U$ is   called \emph{$(\phi,U, F)$-{well}}  if $d_F(x)\leq \phi(x)$, and is called \emph{$(\phi,U, F)$-{tight}}  if $d_F(x)= \phi(x)$.  If   every vertex  in $V_F\setminus U$ is {$(\phi,U,F)$-well}, then $F$ is called a \emph{$(\phi,U)$-{well}} \emph{subgraph}.

   We now present the following useful lemma.

{\begin{lemma} \label{lemma1} Let $G=(V,E)$ be  a non-empty  graph, $\phi$ be a vertex coloring (not necessarily  proper) of $G$ with positive integers,   $I$ be a $\phi$-maximum independent set, and $U\triangleq V\setminus I=\{u_k|k\in[1,p]\}$. If  $u_k$ has at most  $\phi(u_k)-1$ neighbors among $\{u_{j}|j\in[k+1,p]\}$ in $G[U]$ for each $k\in[1,p-1]$, then $G[I,U]$ admits  a   $(\phi,U)$-well subgraph.\end{lemma}

\begin{proof} Clearly, $I\neq \emptyset$, which implies  that $p<|V|$; on the other hand, $I\neq V$ because $G$ is not empty, so   $p>0$.   Then    for each $k\in[1,p]$, $u_k$ admits at least one neighbor $v_k$ in $I$ by the $\phi$-maximum property of $I$.

  If $p=1$, then $U=\{u_1\}$, and
  the graph induced by $v_1u_1$    is a $(\phi,{U})$-well  subgraph. If $p\geq 2$, assume   the conclusion holds for each $k<p$. Suppose, for contradiction,  that there is no $(\phi,{U})$-well subgraph in $G[I,U]$.    Let $G'=G-\{u_1\}$. Observe that
   $I$ is also a $\phi$-maximum independent set in $G'$.  By assumption,
 $G'[I,U\setminus\{u_1\}]$ admits   a
  $(\phi,{U\setminus\{u_1\} })$-well    subgraph $F$. Let $I_F=I\cap V_F$.

   If there exists $x\in N_{G[I,U]}(u_1)$ such that $x\notin I_F$, then the graph induced by $E_F\cup\{xu_1\}$ is a $(\phi,{U})$-well    subgraph of $G$, a contradiction. So $N_{G[I,U]}(u_1)\subseteq I_F$.

 A path $P=u_1x_1y_1x_2y_2\ldots x_{t}y_{t}$ in $G[I,U]$ is called   \emph{$(u_1,F)$-alternating}, if $\{x_jy_{j}|j\in[1,t]\}$ are in $F$. Let $\mathcal P$ be the set of all $(u_1,F)$-alternating paths,    $H$ be the graph induced by $\cup_{P\in \mathcal P}  E_{P}$,   $I_H= I\cap V_H$ and     $U_H=U \cap V_H $. By definition, we have:
\begin{equation} \label{eq1} {\rm each\ star\ component\ of}\  F\    {\rm either\ is\ included\ in}\   H,\  {\rm or\ is\ vertex\text{-}disjoint\ with}\  H.\end{equation}

Let  $P=u_1x_1y_1x_2y_2\ldots x_{t}y_{t}$ be an arbitrary path in $\mathcal P$. If there exists a vertex  $x_{i}$ in $P$ that is not $(\phi,U\setminus\{u_1\}, F)$-{tight}, then the graph induced by $(E_F\setminus\{x_{j}y_{j}|j\in[1,i-1]\})\cup\{y_{j}x_{j+1}|j\in[1,i-1]\}\cup \{u_1x_{1}\}$ is a  $(\phi,{U})$-well subgraph in $G[I,U]$, a contradiction.  Thus: \begin{equation} \label{eq2}
 {\rm each\ vertex\ in}\
I_H\ {\rm is}\   (\phi,U\setminus\{u_1\}, F){\rm\text{-}tight}.
\end{equation}
 If there exists a vertex $y_{i}$ in $P$ that has a neighbor $x$ in $I_F\setminus I_H$, then let  $y$ be a neighbor of $x$ in $F$. Clearly, neither $y_{i}x$ nor $xy$    is  in $H$. However, adding $\{y_{i}x,xy\}$  to the path    $u_1x_{1}y_{1}x_{2}y_{2}\ldots x_{{i}}y_{{i}}$   would yield a  $(u_1,F)$-alternating path that is not included  in $H$,  contradicting   the choice of $H$. Thus:
\begin{equation} \label{eq3}
{\rm each\ vertex\ in}\  U_H\  {\rm has\ no\ neighbor\ in}\  I_F\setminus I_H.
\end{equation}
 If there exists a vertex $y_{i}$ in $P$ that has a neighbor $x$ in $I\setminus I_F$, then the graph induced by $(E_F\setminus\{x_{j}y_{j}|j\in[1,i]\})\cup\{y_{j}x_{j+1}|j\in[1,i-1]\}\cup \{u_1x_{1},y_{i}x\}$ is a  $(\phi,{U})$-well subgraph in $G[I,U]$, a contradiction. Thus,
 \begin{equation} \label{eq4}
 {\rm each\ vertex\ in}\
U_H\ {\rm has\ no\ neighbor\ in}\ I\setminus I_F.
\end{equation}

Suppose $U_H=\{u_{i_j}|j\in[1,\beta]\}$ such that $1=i_1< i_2< \ldots  <   i_\beta$. For each $j\in [1,\beta]$, define   $N^+(u_{i_j})=\{u_{i_j}\}\cup (N_G(u_{i_j})\cap \{u_{i_l}|l> j\})$. Choose  an independent set $Y_0$ in $U_H$ as   follows.
\begin{kst}\item   Let $y_1=u_{i_1}$.
\item \vspace{1mm}For $s\geq 2$, choose $y_s=u_{i_{l_s}}$ such that $i_{l_s}$ is the smallest integer satisfying   $u_{i_{l_s}}\notin\cup_{t<s}N^+(y_t)$.
\end{kst}
   Let $\gamma$ be the least integer such that   $\cup_{t=1}^\gamma N^+(y_t)=U_H$, and let $Y_0=\{y_t|t\in[1,\gamma]\}$.   Then $Y_0$ is an independent set in $G$ contained in  $U_H$. On one hand, by (\ref{eq1}) and (\ref{eq2}), we have:
   \begin{equation*}
   \sum_{x\in I_H}\phi(x)=\sum_{x\in I_H}d_F(x)=|U_H\setminus \{u_1\}|.
   \end{equation*}
   On the other hand,  we observe that:
   \begin{equation*}
   |U_H|=|\cup_{t=1}^\gamma N^+(y_t)|\leq \sum_{t=1}^\gamma [(\phi(y_t)-1)+1]=\sum_{y\in Y_0} \phi(y),
   \end{equation*}
   where the second inequality follows from the given conditions of the lemma. Thus, we conclude that:
    \begin{equation*}
    \sum_{x\in I_H}\phi(x)<\sum_{y\in Y_0}\phi(y).
     \end{equation*} Let $I^\ast=(I\setminus I_H)\cup Y_0$.   Since $Y_0\subseteq U_H$, by (\ref{eq3}) and (\ref{eq4}), $I^\ast$ is an independent set in $G$. However,   $I^\ast$ has  a greater color sum than  $I$ under $\phi$, contradicting the $\phi$-maximum property of $I$.

Therefore, $G[I, U]$ admits       a $(\phi,U)$-well   subgraph. This completes the proof of Lemma \ref{lemma1}.
\end{proof}}

Let $w_0, w$ be two total weightings of a graph $G=(V,E)$. Denote by  $G_{w_0}$   the weighted graph under $w_0$. The  weighting $w$  is called a  \emph{$Q$-proper total weighting of} $G_{w_0}$  if: $w(z)\in Q\subseteq \mathbb R$ for each $z\in V\cup E$, and $\sigma_{w_0+w}(u)\neq \sigma_{w_0+w}(v)$ for each $uv\in E$.

{
\begin{proof}[\textbf{Proof of Theorem \ref{thm}}] It is sufficient to prove that every total weighted simple graph admits a   $\{0,a\}$-total weighting  for every  $a\in \mathbb R^+$ (where $a$ is regarded as the  list span  when consider the uniform-span $(2,2)$-choosability).

  Let $G=(V,E)$ be an arbitrary  simple graph, $w_0$ be an arbitrary total weighting of $G$ and $a$ be an arbitrary positive real number.   For  $v\in V$, denote by $Z_v=\{v\}\cup\{e|e\ni v\}$, and  define   $S_v=\{\sigma_{w_0}(v)+j\cdot a|j\in[0,d(v)+1]\}$. Let   $S=\cup_{v\in V} S_v =\{q_{1},q_{2},\ldots,q_\xi\}$. For each $i\in[1,\xi]$, let $V_i=\{v|S_v\ni q_{i}\}$.

We partition $V$ as follows.

\begin{kst}
\item    Let $G_1\triangleq G[V_1]$, $\phi_1(v)=0$ for each $v$ in $V_{1}$, and   $I_{1}$ be a maximum independent set in  $G_1$.

\item \vspace{1mm}  For $i\geq 2$, if $V_i-(\cup_{j<i}I_{j})\neq\emptyset$,     let $G_{i}\triangleq G[V_i-  (\cup_{j< i}I_{j})]$; let   \begin{equation} \label{eq5}
\phi_{i}(v)=\frac{q_i-\sigma_{w_0}(v)}{a},
\end{equation}
for   each $v$  in $G_{i}$; and let  $I_{{i}}$ be a $\phi_{i}$-maximum independent set in $G_{i}$.  Otherwise, we set $I_{{i}}=\emptyset$.
 \end{kst}
   By the definitions above, one has the following holds.
   \begin{equation} \label{eq60} {\rm If}\   i\geq 2\ {\rm and}\  V_{G_i}\neq \emptyset,\ {\rm   then\ for}\  v\in V_{G_i}\ {\rm and\  }\  j<i\ {\rm with}\ V_j \ni v,\    v\ {\rm  has\ a\ neighbor\ in}\  I_j, \end{equation}  since   $v$ is in $V_{G_j}\setminus I_j$ and $I_j$ is a  $\phi_{j}$-maximum independent set in $G_{j}$.
    By the definition of $S_v$ and by (\ref{eq5}), (\ref{eq60}), we have:
     \begin{equation} \label{eq6} {\rm each\ vertex}\ v\ {\rm in}\ V_{G_i}\  {\rm has}\  \phi_i(v)\  {\rm neighbors\ in}\   \cup_{j< i}I_j.\end{equation}
      If    $V_{\xi-1}-(\cup_{j<\xi-1}I_j)=\emptyset$, then $V_{\xi-1}-(\cup_{j<\xi-1}I_j)$ is an independent set, since each vertex $v$ in   $V_{\xi-1}-(\cup_{j<\xi-1}I_j)$ has $\phi_{\xi-1}(v)$    neighbors in $\cup_{j<\xi-1}I_j$ by (\ref{eq6}), and  since $\phi_{\xi-1}(v)=d(v)$  by (\ref{eq5}).   Let
   $l$ be the greatest integer such that $I_{l}\neq\emptyset$.

  If $V\setminus(\cup_{i=1}^l I_i)\neq\emptyset$,  let  $v\in V\setminus(\cup_{i=1}^l I_i)$.   By definition, $v\in V_j-(\cup_{i=1}^l I_i)$  for each $j\in S_v$. Suppose $S_v=\{j_1,j_2,\ldots,j_r\}$. Then by (\ref{eq6}), $v$ (which is in $V_{j_r}$) has $\phi_{j_r}(v)$ ($=d(v)+1$ by (\ref{eq5}))  neighbors in $\cup_{k<j_r}I_k$, a contradiction. So $V\setminus(\cup_{i=1}^l I_i)=\emptyset$ and   $V=I_1\oplus I_2\oplus\ldots \oplus I_{l}$,  where $I_i$ is allowed to be empty   for $i\neq 1,l$.

  We would construct a proper  $\{0,a\}$-total weighting $w$  of $G_{w_0}$ such that, for each $i\in[1,l]$ and each   $v\in I_i$,    $\sigma_{w_0+w}(v)=q_i$.

 In the  weighting  process, for $z\in V\cup E$, $z$ is   \emph{light} if $w(z)=0$, and   \emph{heavy} if  $w(z)=a$. For $i\in[1,l]$ and   $v\in I_i$, we say  $v$ is \emph{good}  if $\sigma_{w_0+w}(v)\leq q_i$,   \emph{full}  if $\sigma_{w_0+w}(v)=q_i$, and   \emph{hungry}  if $\sigma_{w_0+w}(v)< q_i$. By the definition,   $v$ becomes full if and only if we  finally add       $\phi_i(v)$  times
$a$ to $Z_v$, by (\ref{eq5}).
 For   each $i\in[1,l]$ and  $j=1,2$, denote by  $I^+_{i,j}=\{v\in I_i|\phi_{i}(v)\geq j\}$ and     $I_{i,0}=\{v\in I_i|\phi_{i}(v)=0\}$. Clearly, at the beginning, each vertex is good, where a vertex $v$    is full if and only if  $v\in I_{i,0}$ for some $i\in[1,l]$ (note that $I_{1,0}=I_1$).
  We will ensure that each vertex remains good throughout the process.

  Step by step for decreasing $i$, we first greedily weight    the     vertices and edges in $G[(\cup_{j> i}I_{j})\cap V_i]$ with $a$; denote by   $U_{i+1}$   the set of  vertices in $(\cup_{j> i}I_{j})\cap V_i$ that remain  hungry after the greedily weighting.
Then we use a $(\phi_{i},U_{i+1})$-well   subgraph $F_{i}$ in $G[I^+_{i,1},U_{i+1}]$ to add $a$ to every vertex  in $U_{i+1}$  by weighting each     edge  in  ${F_i}$ with $a$, without adding too much weight to any vertex in $I_i$. In particular, every
hungry vertex $v$ in $I_j$ gains    $a$    in each of some $\phi_j(v)$ steps, ensuring  that it becomes full after
the process. We would use Lemma \ref{lemma1} to guarantee  the
existence of  $F_{i}$. The  details are as follows.

 \vspace{1mm} For $i=l-1,l-2,\ldots,3,2$, suppose the following conditions hold.
\begin{kst}
 \item[(1.1)]  Each vertex in $V$ is good;
  \item[(1.2)] \vspace{1mm} Each $z$    in $\cup_{v\in\cup_{j<i}I_j}Z_v$ is light;
  \item[(1.3)] \vspace{1mm}   For each $t\geq i+2$,   each vertex $v$ in  $I_t\cap V_i$ (if not empty),  and   each $r\in[i+1,t-1]$ with $q_{r}\in S_v$ (if $r$ exists),  one has that $v$ is either is full, or is  adjacent to a vertex $v_r$ in $I_r$   such that $vv_r$ is heavy.
 \end{kst}

  We go   through the light   edges (if they  exist) in  $G[(\cup_{j>i}I_{j})\cap V_i]$ one by one as follows. Right before we consider an edge, if  both of its endpoints are hungry,   we weight it with   $a$; otherwise, we leave it light. After processing  the last edge,  weight each hungry vertex $(\cup_{j>i}I_{j})\cap V_i$ with  $a$.

     At this point,  let $U_{i+1}$ be the   set of hungry vertices in $(\cup_{j\geq i+1}I_{j})\cap V_i$. If $U_{i+1}\neq\emptyset$, let $U_{i+1}=\{u_{\gamma}|\gamma\in[1,p_{i+1}]\}$, where whenever $u_{\gamma}$ is  in $I_\alpha$ and $u_{{\gamma+1}}$ is  in $I_\beta$, one has that $\alpha\leq \beta$. By our weighting method, each vertex and each edge in   $G[U_{i+1}]$ is heavy.

 Let $u_{\gamma}$ be an arbitrary vertex in $U_{i+1}$ and suppose
$u_{\gamma}$ is in $I_t\cap V_i$ for some $t>i$. By  (1.3) and (\ref{eq60}), ${u_\gamma}$  is incident   to at least $\phi_t(u_\gamma)-\phi_i(u_\gamma)-1$ heavy edges whose other endpoints all lie in $\cup_{i<j<t}I_j$. So $u_\gamma$ receives a weight at least $\sigma_{w_0}(u_\gamma)+(\phi_t(u_\gamma)-\phi_i(u_\gamma)-1)\cdot a+a
=\sigma_{w_0}(u_\gamma)+(\phi_t(u_\gamma)-\phi_i(u_\gamma))\cdot a$. Since $u_\gamma$ is hungry,  $\sigma_{w_0}(u_\gamma)
+(\phi_t(u_\gamma)-\phi_i(u_\gamma))\cdot a
<q_t=\sigma_{w_0}(u_\gamma)+\phi_t(u_\gamma)\cdot a$  by (\ref{eq5}). It follows that $\phi_i(u_\gamma)>0$.

If $I^+_{i,1}= \emptyset$, then the color sum of $I_i$ under $\phi_i$ equals   0. However, $\{u_\gamma\}$ is an independent set in $G_i$ that has the color sum at least 1 (so greater than that of $I_i$) under $\phi_i$,   contradicting   the choice of $I_i$. So $I^+_{i,1}\neq \emptyset$.  Note that $\sum_{v\in I^+_{i,1}}\phi_i(v)=\sum_{v\in I_{i}}\phi_i(v)$ by definition. On one hand, if there exists an independent set $I^\ast$ in $G[I^+_{i,1},U_{i+1}]$ such
that $\sum_{v\in I^\ast}\phi_i(v)>\sum_{v\in I^+_{i,1}}\phi_i(v)$, then $I^\ast$
is an independent set in $G_i$ with a greater $\phi_i$-sum than $I_i$, a contradiction to
the choice of $I_i$. So $I^+_{i,1}$ has the greatest $\phi_i$-sum among all independent set in $G[I^+_{i,1},U_{i+1}]$.
On the other hand, if some $u$ in $U_{i+1}$   has no neighbor in $I^+_{i,1}$, then $I^+_{i,1}\cup\{u\}$ is an independent set
in $G_i$ with a greater $\phi_i$-sum than $I_i$, contradicting
the choice of $I_i$. Thus, by definition, $I^+_{i,1}$ is a $\phi_i$-maximum independent set
in $G[I^+_{i,1},U_{i+1}]$.  Moreover, since $u_\gamma$ is hungry, $u_\gamma$ is incident to at most $(\phi_t(u_\gamma)-1)-(\phi_t(u_\gamma)-\phi_i(u_\gamma))=\phi_i(u_\gamma)-1$ heavy edges in $G[(\cup_{j>t}I_j)\cap V_i]$.
  That is,  $u_\gamma$ admits at most $\phi_{i}(u_\gamma)-1$ neighbors in $\{{u_\zeta}|\zeta>\gamma\}$, since $I_t$ is an independent set. By the arbitrary choice of $u_\gamma$, and by Lemma \ref{lemma1}, there exists a $(\phi_{i},U_{i+1})$-well   subgraph $F_{i}$ in $G[I^+_{i,1},U_{i+1}]$. Weight   each edge in $F_{i}$ with $a$. At this point,  each vertex $v$ in $I_i$ is still good, since $d_{F_i}(v)\leq \phi_i(v)$ (which implies that $v$ is incident     to at most $\phi_i(v)$ heavy edges). Weight  each hungry vertex in $I_i$ with $a$.

  Now,  the following conditions hold.
  \begin{kst}
 \item[(2.1)]  Each vertex in $V$ is good;
  \item[(2.2)] \vspace{1mm} Each $z$    in $\cup_{v\in\cup_{j<i-1}I_j}Z_v$ is light;
  \item[(2.3)] \vspace{1mm}   For each $t\geq i+1$,   each vertex $v$ in  $I_t\cap V_{i-1}$ (if not empty),  and   each $r\in[i,t-1]$ with $q_{r}\in S_v$ (if $r$ exists),  one has that $v$ is either is full, or is  adjacent to a vertex $v_r$ in $I_r$   such that $vv_r$ is heavy.
 \end{kst}

After completing the iterative steps, clearly, each vertex in {$V\setminus(\cup_{i\geq 1}I^+_{i,2})$} is full. Let $t$ be an arbitrary integer in $[3,l]$ (note that  $I^+_{1,2}, I^+_{2,2}=\emptyset$). If $I^+_{t,2}\neq\emptyset$, let $v$ be an arbitrary vertex in $I^+_{t,2}$. Then by (2.3) and (\ref{eq6}), $v$ is incident to at least $\phi_{t}(v)-1$ heavy  edges. Moreover, $v$ is heavy  by our weighting method. Thus, $v$ is full. Since
$t$ and $v$ were chosen arbitrarily, every vertex in {$\cup_{i\geq 1}I^+_{i,2}$} is   full. Therefore, we have   obtained a proper $\{0,a\}$-total weighting   for $G_{w_0}$.

This completes our proof.
\end{proof}}

\noindent\textbf{Acknowledgments}

The   first author is supported by NSFC (No. 11701195), Fundamental Research Funds for the Central
Universities (No. ZQN-904) and Fujian Province University Key Laboratory of Computational Science,
School of Mathematical Sciences, Huaqiao University, Quanzhou, China (No. 2019H0015).


\begin{thebibliography}{99}
\small \setlength{\itemsep}{.8mm}

\bibitem{AddarioBerry} L. Addario-Berry, K. Dalal, C. McDiarmid, B. Reed, A. Thomason, Vertex-colouring edge weightings, Combinatorica 27 (2007) 1-12.

     \bibitem{AddarioBerry1} L. Addario-Berry, K. Dalal, B.A. Reed, Degree constrained subgraphs, Discrete Appl. Math. 156 (7)
(2008) 1168-1174.



\bibitem{AlonWei} N. Alon, F. Wei, Irregular subgraphs, Combin. Probab. Comput. 32 (2) (2023) 269-283.





      \bibitem{Bartnicki09}  T. Bartnicki, J. Grytczuk, S. Niwczyk, Weight choosability of graphs, J. Graph Theory 60 (2009)
242-256.

 \bibitem{Baudona} O. Baudon, J. Bensmail, J. Przyby{\l}o, M. Wo\'{z}niak, On decomposing regular graphs into locally
irregular subgraphs, Eur. J. Comb. 49 (2015) 90-104.


      \bibitem{Bensmail20231}    J. Bensmail, H. Hocquard, D. Lajou, \'{E}. Sopena, A proof of the multiplicative 1-2-3 conjecture, Combinatorica 43 (2023) 37-55.



\bibitem{Bondy} J.A. Bondy, U.S.R. Murty, Graph Theory, in: Graduate Texts in Mathematics, vol. 244, Springer,  New York, 2008.

       \bibitem{Cao} L. Cao, Total weight choosability of graphs: towards the 1-2-3-conjecture, J. Comb. Theory, Ser. B 149 (2021) 109-146.

    \bibitem{Chartrand} G. Chartrand, M.S. Jacobson, J. Lehel, O. Oellermann, S. Ruiz, F. Saba, Irregular networks, Congr.
Numer. 64 (1988) 197-210.




\bibitem{dq}  K. Deng, H. Qiu, The 1-2 conjecture holds for regular graphs, J. Comb. Theory, Ser. B 174 (2025) 207-213.


 \bibitem{Dudek} A. Dudek, D. Wajc, On the complexity of vertex-coloring edge-weightings, Discrete Math. Theor.
Comput. Sci. 13 (3) (2011) 45-50.

 \bibitem{Faudree}R.J. Faudree, J. Lehel, Bound on the irregularity strength of regular graphs, Colloq. Math. Soc.
J\'{a}nos Bolyai 52 (1988) 247-256.

     \bibitem{FoxLuoPham}    J. Fox, S. Luo, H.T. Pham, On random irregular subgraphs, Random Struct. Algorithms 64 (2024) 899-917

 \bibitem{Grytczuk}  J. Grytczuk, From the 1-2-3 conjecture to the Riemann hypothesis, Eur. J. Comb. 91 (2020), Paper 103213, 10pp.

 \bibitem{Kalkowski1} M. Kalkowski, Ph.D. thesis, 2009.

 \bibitem{Kalkowski} M. Kalkowski, M. Karo\'{n}ski, F. Pfender, Vertex-coloring edge-weightings: towards the 1-2-3-conjecture, J. Comb. Theory, Ser. B 100 (2010) 347-349.






    \bibitem{Kalkowski2011}  M. Kalkowski, M. Karo\'{n}ski, F. Pfender, A new upper bound for the irregularity strength of graphs.
SIAM J. Discrete  Math. 25(3) (2011)  1319-1321.

  \bibitem{Kalkowski2017}  M. Kalkowski, M. Karo\'{n}ski, F. Pfender,  The 1-2-3-conjecture for hypergraphs. J. Graph Theory
85(3) (2017)  706-715.




\bibitem{Thomason}  M. Karo\'{n}ski, T. {\L}uczak, A. Thomason, Edge weights and vertex colours, J. Comb. Theory, Ser. B
91 (2004) 151-157.

 \bibitem{Keusch}R. Keusch, Vertex-coloring graphs with 4-edge-weightings, Combinatorica 43 (2023) 651-658.

      \bibitem{Keusch1}R. Keusch, A solution to the 1-2-3 conjecture, J. Comb. Theory, Ser. B 166 (2024)  183-202.





  \bibitem{Luzar2025}  B. Lu\v{z}ar, J. Przyby{\l}o, R. Sot\'{a}k, Degree-balanced decompositions of cubic graphs, Eur. J. Combin. 128 (2025), Paper No. 104169, 10 pp.

 \bibitem{MaXie}  J. Ma, S. Xie, Finding irregular subgraphs via local adjustments, J. Combin. Theory, Ser. B 174 (2025) 71-98.

    \bibitem{Przybylo3} J. Przyby{\l}o,  A note on neighbour-distinguishing regular graphs total-weighting, Electron. J. Comb. 15(1) (2008),   Note 35, 5pp.


   \bibitem{Przybylo}   J. Przyby{\l}o, The 1-2-3 conjecture almost holds for regular graphs, J. Comb. Theory, Ser. B 147
(2021) 183-200.

     \bibitem{Przybylo1} J. Przyby{\l}o, The 1-2-3 conjecture holds for graphs with large enough minimum degree, Combinatorica 42 (2022) 1487-1512.

    \bibitem{Przybylo2} J. Przyby{\l}o, M. Wo\'{z}niak, On a 1-2 conjecture, Discrete Math. Theor. Comput. Sci. 12 (2010)
101-108.




  \bibitem{Przybylo4} J. Przyby{\l}o, M. Wo\'{z}niak, Total weight choosability of graphs, Electron. J. Comb. 18 (1) (2011),  Paper 112, 11pp.





\bibitem{Przybylo6} J. Przyby{\l}o, F. Wei, Short proof of the asymptotic confirmation of the Faudree-Lehel conjecture,
Electron. J. Comb. 30 (4) (2023)  P4.27.


\bibitem{Przybylo7} J. Przyby{\l}o, F. Wei, On the asymptotic confirmation of the Faudree-Lehel conjecture for general
graphs, Combinatorica 43 (2023) 791-826.



     \bibitem{Seamone}    B. Seamone, The 1-2-3 conjecture and related problems: a survey, arXiv:1211.5122 [math.CO], 2012.




\bibitem{Thomassen1} C. Thomassen, Y. Wu, C.-Q. Zhang, The 3-flow conjecture, factors modulo $k$, and the 1-2-3-conjecture, J. Comb. Theory, Ser. B 121 (2016) 308-325.



  \bibitem{Wang} T. Wang, Q. Yu, A note on vertex-coloring 13-edge-weighting, Front. Math. China 3 (2008)
581-587.



     \bibitem{Wong}  T.-L. Wong, X. Zhu, Total weight choosability of graphs, J. Graph Theory 66 (2011) 198-212.

       \bibitem{Wong1}    T.-L. Wong, X. Zhu, Every graph is (2, 3)-choosable, Combinatorica 36 (2014) 121-127.


 \bibitem{Zhong} L. Zhong, The 1-2-3-conjecture holds for dense graphs, J. Graph Theory 90 (2018) 561-564.

     \bibitem{Zhu}   X. Zhu, Every nice graph is (1, 5)-choosable, J. Comb. Theory, Ser. B 157 (2022) 524-551.


\end{thebibliography}
\end{document}